\documentclass[12pt,leqno]{amsart}
\usepackage{amssymb,amsthm,amsmath,latexsym}
\newtheorem{theorem}{\sc Theorem}[section]
\newtheorem{lemma}[theorem]{\sc Lemma}
\newtheorem{proposition}[theorem]{\sc Proposition}
\newtheorem{corollary}[theorem]{\sc Corollary}

\newtheorem{rem}[theorem]{\sc Remark}

\date{}

\title[Maximal covers]{Maximal covers of finite groups}
\author[Lima]{Igor Lima}
\address{Departamento de Matem\'atica, Universidade de Bras\'ilia,
Brasilia-DF, 70910-900 Brazil }
\email{igor.matematico@gmail.com}
\author[Bastos]{Raimundo Bastos}
\address{Departamento de Matem\'atica, Universidade de Bras\'ilia,
Brasilia-DF, 70910-900 Brazil }
\email{bastos@mat.unb.br}
\thanks{The second author was supported by Funda\c c\~ao de Apoio \`a Pesquisa do Distrito Federal (FAPDF), Brazil, Grant: 0193.001344/2016.}
\author[Rog\'erio]{Jos\'e R. Rog\'erio}
\address{Departamento de Matem\'atica, Universidade Federal do Cear\'a, Fortaleza - CE, Brazil}
\email{roberiorogerio@yahoo.com.br}

\begin{document}

\begin{abstract} 
Let $\lambda(G)$ be the maximum number of subgroups in an irredundant covering of the finite group $G$. We prove that if $G$ is a group with $\lambda(G) \leqslant 6$, then $G$ is supersolvable. We also describe the structure of the groups $G$ with $\lambda(G)=6$. Moreover, we show that if $G$ is a group with $\lambda(G) \leqslant 30$, then $G$ is solvable.   
\end{abstract}

\maketitle

\section{Introduction}

Throughout this paper all groups are finite. Let $G$ be a group. The subgroups $X_{1},X_{2},...,X_{n}$, $n\geq 2$, form a \textit{covering} of $G$ if 
\begin{equation*}
\bigcup_{i=1}^{n}X_{i}=G\,.
\end{equation*}%
A covering is said to be \textit{irredundant} if none of the subgroups can
be removed, i.e. 
\begin{equation*}
X_{i}\nsubseteq \bigcup_{j\neq i}^{n}X_{j}\,,
\end{equation*}%
for all $i=1,2,...,n$. If $\{X_i\}_{i=1}^{n}$ is an irredudant covering of $G$ by subgroups, it is natural to ask what information about $G$ can be deduced from properties of the subgroups $X_i$ and the size $n$. 

The study of the coverings of a group by its subgroups dates back to 1926, when G. Scorza \cite{S} proved that a
group $G$ admits an irredundant covering by $3$ subgroups if and only if it
has a normal subgroup $N$ such that $G/N\cong C_{2}\times C_{2}$. Since
then, several authors have been working on problems involving group
coverings in various settings. J. Cohn \cite{Cohn} defined $\sigma (G)$ as
the minimal number of subgroups in an irredundant covering of the non-cyclic
group $G$. He outlined the basic properties of $\sigma (G)$, classified the
groups with $\sigma (G)=3,4,5$ and observed other intriguing properties of
this function, for instance, that apparently there were no groups $G$ for
which $\sigma (G)=7$. He went further and conjectured that if $G$ is a
non-cyclic soluble group then $\sigma (G)=p^{a}+1$, with $p$ prime
and $a\in \mathbb{N}$. Both of these conjectures were later proved by M.
Tomkinson in \cite{Tomk1}. More recently, A. Abdollahi et al. \cite{AAAH} gave
a complete classification of the groups with $\sigma(G)=6$, while J. Zhang 
\cite{JZhang} proved the non-existence of finite groups with $\sigma
(G)=11,13 $ and showed that $\sigma \bigl(PSL(2,7)\bigr)=15$. Other papers
that contributed to this theory of minimal coverings over the last decades
are \cite{Abdo1}, \cite{Abdo2}, \cite{BFS}, \cite{Bry1} and \cite{Marioti}. 

In \cite{B}, R. Brodie considered a slightly different problem. He classified the groups $G$ that have \textit{exactly one
irredundant covering} by proper subgroups (see also \cite{JR,JRR} for more details). In \cite{JRR}, appeared another extremal variant of the covering problem. For a group $G$, define the function $\lambda(G)$ as the \textit{maximum}
number of subgroups in an irredundant covering of $G$. In the same paper, the third named author proved the basic properties of $\lambda(G)$, also presented a description of groups with $\lambda(G) =3,4,5$ and study the structure of groups $G$ admitting only one-sized coverings, i.e., such that $\sigma(G) = \lambda(G)$. In \cite{GL}, M. Garonzi and A. Lucchini gave a complete description of groups admitting only one-sized covering.  

In the present paper we want to study the structure of group $G$ in terms of the number $\lambda(G)$. We obtain the following related results. \\

{\noindent}{\bf Theorem A.} {Let $G$ be a group with $\lambda(G) \leqslant 6$. Then $G$ is supersolvable.} \\ 

Note that groups with $\lambda(G)=7$ can be non-supersolvable. For instance, $\lambda(A_4)=7$. In the next two results, we classify the structure of groups $G$ with $\lambda(G)=6$. \\

{\noindent}{\bf Theorem B.}{ Let $G$ be a nilpotent group with $\lambda(G) = 6$. Then $G \cong P \times C_n$, where $P$ is a Sylow $p$-subgroup and $p$ does not divides $n$. Moreover, $G$ is isomorphic to one of the following: }
\begin{itemize}
 \item[$(a)$] $C_4 \times C_4 \times C_n$, where $n$ is an odd positive integer;
 \item[$(b)$] $C_5 \times C_5 \times C_n$, where $5$ does not divide $n$;
 \item[$(c)$] $C_3 \times C_9 \times C_n$, where $3$ does not divide $n$;
 \item[$(d)$] $R_3 \times C_n$, where $R_3 \cong \langle a,b \mid a^9=1=b^3, a^b=a^4\rangle$ and $3$ does not divides $n$; 
 \item[$(e)$] $G \cong C_2 \times C_{16} \times C_n$, where n is an odd positive integer;
 
 \item[$(f)$] $G \cong T_5 \times C_n$, where $T_5 = \langle a,b \mid a^{16}=1=b^2, \ a^b = a^9\rangle$ and $n$ is an odd positive integer. 
\end{itemize}

{\noindent}{\bf Theorem C.}{ Let $G$ be a non-nilpotent group with $\lambda(G) = 6$. Then $G$ is isomorphic to one of the following: }
\begin{itemize}
\item[(a)] $\langle a,b \mid \ a^5 = 1 = b^n, \ a^b = a^4\rangle$, where $4^n \equiv 1$ $( \text{mod \ }5)$ and $5 \nmid n$;
\item[(b)] $\langle a,b \mid \ a^5 = 1 = b^n, \ a^b = a^2\rangle$, where $2^n \equiv 1$ $(\text{mod \ } 5)$ and $5 \nmid n$ and $4 \mid n$; 
\item[(c)] $\langle a,b \mid \ a^5 = 1 = b^n, \ a^b = a^3\rangle$, where $3^n \equiv 1$ $(\text {mod \ } 5)$, $5 \nmid n$ and $4 \mid n$;
\item[(d)] $\langle a,b \mid a^{3n} = 1 = b^{6n}, a^b = a^{n+1}, a^3 = b^6 \rangle$, where $3 \mid n$ or $3 \mid (n+2)$;
\item[(e)] $\langle a,b \mid a^{3n} = 1 = b^{6n}, a^b = a^{2n+1}, a^3 = b^6 \rangle$, where $3 \mid n$ or $3 \mid (n+1)$; 
\item[(f)] $\langle a,b \mid a^{3n} = 1 = b^{6n}, a^3 = b^6, a^b = a^{n+1}\rangle$, where $2 \mid n$, $3 \mid n$ or $3 \mid (n+2)$;
\item[(g)] $\langle a,b \mid a^{3n} = 1 = b^{6n}, a^3 = b^6, a^b = a^{2n+1}\rangle$, where $2 \mid n$ and $3 \nmid (n+2)$.
\end{itemize} 

In the next result we provides a solubility criterion for a group $G$ in terms of the $\lambda(G)$. \\

{\noindent}{\bf Theorem D.} {Let $G$ be a group with $\lambda(G) \leqslant 30$. Then $G$ is solvable.} \\ 

Note that groups with $\lambda(G)=31$ can be non-solvable. For instance, $\lambda(A_5)=31$ (see also Remark \ref{rem.solvable} for more examples). 

The paper is organized as follows. In the next section we collect results in the context of irredundant covers of groups. In the third section we present the proof of Theorem A, where we give a supersolubility criterion for groups in terms of the $\lambda(G)$. Section 4 is devoted to the description of groups with exactly $\lambda(G)=6$ (Theorems B and C). In Section 5 we present two solvability criteria for groups in terms of the number of cyclic subgroups of $G$ (Theorem D and Corollary \ref{cor.31}).      
 	 	 
\section{Preliminary results}
We say that the subgroup $\left\langle x\right\rangle $\ of the group $G$\
is \textit{maximal cyclic} if there is no cyclic subgroup $\langle y\rangle 
$ in $G$ such that $\langle x\rangle <\langle y\rangle $. In the first lemma we collect some results which can be found in \cite{JRR}.

\begin{proposition} (\cite[Propositions 4 and  5]{JRR})
\label{prop.basic} Let $G$ and $W$ be finite groups.

\begin{enumerate}
\item[(i)] If $H_1, \ldots, H_n$ be the maximal cyclic subgroups of $G$, then $
\cup_{i=1}^{n} H_{i}$ is an irredundant covering of $G$ and $\lambda (G)=n$. Moreover, $\cup_{i=1}^{n}
H_i$ is the unique covering by proper subgroups
of $G$ if and only if $H_i$ is a maximal subgroup of 
$G$, $i=1,\ldots, n$
\item[(ii)] For any subgroup $H\leq G,~$we have $\lambda (H)\leq \lambda (G).$
\item[(iii)] If $M \trianglelefteq G$, then $\lambda \big(\frac{G}{M}\big)%
\leq \lambda (G)$. For $N=\bigcap\limits_{i=1}^{\lambda }H_{i}$, where $%
\lambda =\lambda (G)$ and $H_{1},\,H_{2},...,\,H_{\lambda}$ are the maximal cyclic
subgroups of $G$, we have $N$ is a central subgroup of $G$ and  $\lambda \big(\frac{G}{N}\big)=\lambda (G)$.
\item[(iv)] If each maximal cyclic is a maximal subgroup, then $\lambda %
\big(\frac{G}{\Phi (G)}\big)=\lambda (G)$, where $\Phi (G)$ is the Frattini
subgroup of $G$.
\item[(v)] We have $\lambda (G\times W)\geq \lambda (G)\lambda (W)$.
Equality is true if $\gcd\left( \left\vert G\right\vert ,\left\vert
W\right\vert \right) =1.$
\end{enumerate}
\end{proposition}

For every positive integer $n \geqslant 3$, denote by $f(n)$ the largest index $|G:D|$, where $G$ is a group with an irredundant cover $X_1, \ldots, X_n$ and which $D = \cap_{i=1}^{n} X_i$. In this context it is known the following values to the function $f$.  

\begin{proposition} \label{Tom}
\begin{itemize}
 \item[(a)](\cite{S}) $f(3) = 4$;
 \item[(b)] (\cite{T}) $f(4) = 9$;
 \item[(c)] (\cite{BFS}) $f(5) = 16$;
 \item[(d)] (\cite[Theorem D]{AAAH}) $f(6) = 36$. 
\end{itemize}
\end{proposition}

For arbitrary numbers there are bounds for the values of the function $f$, but the precise value is unknown (for more details see \cite{T}).     

\begin{lemma} \label{lem.Cohn} (J. Cohn \cite{Cohn}) 
\begin{itemize}
\item[(a)] If $\bigcup\limits_{i=1}^{n}X_{i}$ is an
irredundant covering of the group $G$, then $\left\vert G\right\vert
\leq \sum\limits_{i=2}^{n}\left\vert X_{i}\right\vert .$
\item[(b)] Let $n$ be a positive integer and $G$ a group with $\lambda(G)=n$. Then there exists a cyclic subgroup $H$ of $G$ whose index is at most $n-1$.  
\end{itemize}
\end{lemma}

We need the following result, due to Bryce, Fedri and Serena \cite{BFS}.  

\begin{lemma} \label{BFS}
Let $X_1, X_2, \ldots, X_m$ be an irredundant cover of a group $G$ whose intersection is $D$. If $p$ is a prime, $x$ a $p$-element of $G$ and $|\{i \ : \ x \in X_i\}|=n$, then either $x \in D$ or $p \leqslant m-n$. 
\end{lemma}

The following remark give us $\lambda(G)$ for particular abelian groups (see \cite[Proposition 6]{JRR}).  

\begin{rem} \label{rem.abelian} Let $k \geqslant 1$. Then $\lambda (C_{p}\times
C_{p^{k}})=kp-k+2$; $\lambda (C_{p}\times C_{p}\times
C_{p^{k}})=kp^{2}+p-k+2 $; $\lambda (C_{p^{2}}\times
C_{p^{k}})=(k-1)p^{2}-(k-3)p$ and $\lambda (C_{p}\times C_{p^{3}}\times
C_{p^{k}})=(k-1)p^{3}+2p^{2}-(k-2)p$. 
\end{rem}

\section{Supersolvability criterion}

As usual, for a group $G$ we denote by $\pi(G)$ the set of prime divisors of the order of elements of $G$. If a group $G$ has $\pi(G) = \pi$, then we say that $G$ is a $\pi$-group. Throughout the rest of the paper for every group $G$ with $\lambda(G) = m$, the subgroups $H_1, H_2, \ldots, H_m$ stands the set of maximal cyclic subgroups of $G$, $N = \cap_{i=1}^{m} H_i$ and $\overline{G} = G/N$. In particular, we can assume that $|H_i| \geqslant |H_{i+1}|$, for $i\in \{1,2,\ldots,m\}$.  

The goal of this section is to obtain a supersolvability criterion for groups in terms of $\lambda(G)$. The key result is restrict the possibilities of the order $\vert \overline{G}\vert$, when $\lambda(G)=6$. 

\begin{lemma} \label{bound}
Let $G$ be a group with $\lambda(G) = 6$. Then 
\begin{itemize}
 \item[$(a)$] The order $\vert \overline{G}\vert$ is at most $36$; 
 \item[$(b)$] The quotient group $\overline{G}$ is a $\{2,3,5\}$-group. 
\end{itemize}
\end{lemma}
\begin{proof} 
(a). Since the maximal cyclic subgroups $H_1,H_2,H_3,H_4,H_5,H_6$ form an irredundant cover of $G$, we conclude that $\vert \overline{G} \vert \leqslant 36$ (Proposition \ref{Tom}$(d)$). \\ 
 
\noindent{(b)}. If $p \in \pi(\overline{G})$ and $x \in G$ such that the order $\vert xN\vert = p$, then $p \leqslant m-n \leqslant 6-n \leqslant 6$ (Lemma \ref{BFS}), where $n = |\{i \ : \ x \in H_i\}|$. Consequently, $\pi(\overline{G}) \subseteq \{2,3,5\}$. In particular, the order $\vert \overline{G} \vert \leqslant 36$ and $\overline{G}$ is a $\{2,3,5\}$-group (Lemma \ref{bound}), which completes the proof. 
\end{proof}

\begin{lemma} \label{cyclic}
Let $G$ be a group with $\lambda(G)=6$. Then $$|\overline{G}| \not\in \{2,3,5,7,11,13,15,17,19,23,29\}.$$
\end{lemma}

\begin{proof}
By Proposition \ref{prop.basic} $(ii)$, $\lambda(G) = \lambda(\overline{G})$. From this we can deduce that $\overline{G}$ cannot be cyclic. In particular,  $$|\overline{G}| \not\in \{2,3,5,7,11,13,15,17,19,23,29\},$$ which completes the proof. 
\end{proof}

\begin{lemma} \label{lem.tec}
Let $G$ be a group with $\lambda(G)=6$. Then $$\vert \overline{G}\vert \not\in \{4,6,8,12,24,30\}.$$ 
\end{lemma}

\begin{proof}  Without loss of generality we can assume $\overline{G}$ is not cyclic (Proposition \ref{prop.basic} $(ii)$). \\

\noindent{(Case $\vert \overline{G}\vert = 4$).} Then $\overline{G}$ is isomorphic to $C_2 \times C_2$. But $\lambda(C_2 \times C_2) = 3$ and therefore this case does not occur. \\    

 \noindent{(Case $\vert \overline{G}\vert = 6$).} Then $\overline{G}$ is isomorphic to $S_3$. However, $\lambda(S_3) = 4$ and we can discart this case. \\ 
 
\noindent{(Case $\vert \overline{G}\vert = 8$).} Then $\overline{G}$ is isomorphic to one of the following groups $C_2 \times C_4, C_2 \times C_2 \times C_2, D_4$ or $Q_8$. Note that $\lambda(C_2 \times C_4) = 4$, $\lambda(C_2\times C_2 \times C_2) = 7$, $\lambda(Q_8) = 3$ and $\lambda(D_4)=5$. Hence this case does not occur. \\ 

\noindent{(Case $\vert \overline{G}\vert = 12$).} Then $\overline{G}$ is isomorphic to one of the following $C_2 \times C_2 \times C_3, D_6$, $A_4$ or  $T = \langle a,b \vert \ a^3=b^4=1, \ a^b=a^{-1} \rangle$. Thus, $\lambda(C_2 \times C_2 \times C_3) = 3$, $\lambda(D_6) = \lambda(A_4) = 7$ and $\lambda(T) = 4$. Consequently, this case does not happen. \\    

\noindent{(Case $\vert \overline{G} \vert = 24$).} If $\overline{G}$ is abelian and noncyclic, then $\overline{G}$ is isomorphic to one of the following $$C_2 \times C_2 \times C_2 \times C_3, C_2 \times C_4 \times C_3.$$ By Remark \ref{rem.abelian}, $\lambda(C_2 \times C_2 \times C_2 \times C_3) = 7$ and $\lambda(C_2 \times C_4 \times C_3) = 4$. 
 
Now, assume that $\overline{G}$ is non-abelian. Hence $\overline{G}$ is isomorphic to one of the following 
\begin{center}
$D_{12}, S_4, C_4 \times S_3, SL(2,3), C_3 \times Q_8, C_2 \times D_6, C_2 \times A_4, C_3 \times D_4, L, T, Q, M$, 
\end{center} 
  where $L=\langle a,b \vert \  a^3=1=b^8, a^b=a^{-1}\rangle, T=\langle a,b,c \vert \   a^2=1=b^2=c^3=(ba)^4, c^a=c^{-1}, c^b=c \rangle$, $Q=\langle a,b \vert \  a^{12}=1=b^4, a^6=b^2, a^b=a^{-1} \rangle$ and $M = \langle a,b,c \vert \ 1=a^2=b^3=c^4, bcb=c, aca=c, b^a=b \rangle$. Thus, $\lambda(D_{12}) = 13$, $\lambda(S_4) \geqslant \lambda(A_4) = 7$, $\lambda(C_4 \times S_3)=12$, $\lambda(C_3 \times D_4) = \lambda(D_4) = 5$, $\lambda(C_3 \times Q_8) = \lambda(Q_8) = 3$. We have $\lambda(C_2 \times D_6) \geqslant 7$, because $\lambda(D_6) = 7$. We have $\lambda(S_4), \lambda(SL(2,3)), \lambda(C_2 \times A_4) \geqslant 7$, because $SL(2,3)/Z(SL(2,3)) = PSL(2,3) \cong A_4$, $A_4$ is a subgroup of $A_4 \times C_2, S_4$ and $\lambda(A_4) = 7$. Moreover, $\lambda(L) = 4$, $\lambda(T) = 12$, $\lambda(Q)=7$ and $\lambda(M)=9$. Thus, we can assume that $\overline{G}$ cannot have order $24$. \\
 
\noindent{(Case $\vert \overline{G} \vert = 30$).} If $\vert \overline{G} \vert = 30$ then $\overline{G}$ is isomorphic to one of the following $D_{15}$, $S_3 \times C_5$, $D_5 \times C_3$. Note that $\lambda(S_3 \times C_5) = \lambda(S_3) = 4$, $\lambda(D_{15}) = 16$ and $\lambda(D_5 \times C_3) = \lambda(D_5) = 6$. \\

So, it remains to exclude the case when $\overline{G} = D_5 \times C_3$. Firstly, we will prove that $N = \cap_{i=1}^{6} H_i = Z(G)$. It is clear that $N \leqslant Z(G)$. It remains to prove that $Z(G) \leqslant N$. There is no lost in  assume that $\vert G:H_1 \vert = 2$ and $\vert G : H_i \vert = 5$, $i=2,3,4,5,6$. Thus, every maximal cyclic subgroup is also maximal and so $G$ is unique covered. Since $G$ is non-nilpotent, it follows that $\vert G/Z(G)\vert = pq$ and $\langle x,Z(G)\rangle = \langle x \rangle$ for every $x \in G$ \cite[Theorem 2.5]{JR}. In particular, $Z(G)$ is contained in every maximal cyclic subgroup $H_i$. Therefore $$Z(G) \leqslant \displaystyle\bigcap_{i=1}^{6}H_i = N$$ and  $pq = \vert G/Z(G) \vert = \vert \overline{G} \vert =30,$ a contradiction. The proof is complete.       
\end{proof}

Combining Lemmas \ref{bound}, \ref{cyclic} and \ref{lem.tec}, one obtains

\begin{proposition} \label{prop.list}
 Let $G$ be a group with $\lambda(G) = 6$. Then the order $\vert \overline{G} \vert \in \{10,16,18,20,25,27,32,36\}$. 
\end{proposition}

We obtain the list of all groups $G$ with $\lambda(G)=6$ and order at most $36$ (using GAP \cite{GAP}). This next result will be needed in the proof of Theorem B.     

\begin{lemma} \label{lem.grp.6}
Let $G$ be group with $\lambda(G)=6$. If $G$ has order at most $36$, then $G$ is isomorphic to one of the following $D_5$, $C_4 \times C_4$, $S_3 \times C_3$, $\langle a,b \mid a^5= 1 = b^4, a^b =a^{-1} \rangle$, $\langle a,b \mid a^5 = 1 = b^4, a^b = a^{2}, [a,b]=a\rangle $ $C_5 \times C_5$, $R_3 = \langle a,b \mid a^9 = 1 = b^3, \ a^b = a^4 \rangle$, $C_3 \times C_9$, $C_{16} \times C_2$, $D_5 \times C_3$, $T_5= \langle a,b \mid a^{16}=1=b^2, \ a^b = a^9\rangle$ or $\langle a,b,c \mid a^3=1=b^3=c^4, a^c = a^2, b^c =b, a^b=a\rangle$. 
\end{lemma}

We are ready to prove Theorem A. \\ 

{\noindent}{\bf Theorem A.} {Let $G$ be a group with $\lambda(G) \leqslant 6$. Then $G$ is supersolvable.}

\begin{proof} 
Suppose that $\lambda(G) \leqslant 5$. By \cite[Theorem 1]{JRR}, the group $G$ is supersolvable. 

Now, we can assume that $\lambda(G)=6$. Let $H_1, \ldots, H_6$ be the set of all maximal cyclic subgroups of $G$ and $N = \cap_{i=1}^{6} H_i$. By Proposition \ref{prop.list} and Lemma \ref{lem.grp.6}, the quotient $\overline{G} = G/N$ is supersolvable. By Proposition \ref{prop.basic}, $N$ is a cyclic normal subgroup of $G$. Therefore, the group $G$ is supersolvable. 
\end{proof}

\begin{rem}
Note that the above bound cannot be improved. For instance, $\lambda(A_4 \times C_n)=7$, where $n$ and $6$ are coprime numbers. 
\end{rem}

\section{Groups with $\lambda(G)=6$}
In this section we describe the structure of groups with exactly $\lambda(G)=6$ (Theorems B and C). First we will examine the Theorem B in the context of Dedekind groups. Recall that a group is said Dedekind if every subgroup is normal (see \cite[5.3.7]{Rob} for more details). Moreover, for abbreviation, we follow the notation of \cite{JRR}. More precisely,  

\begin{itemize}
\item $D_{n}=\langle
a,b \mid a^{n}=1=b^{2},\,a^{b}=a^{-1}\rangle$;
\item $Q_{2^{n}}=\langle
a,b \mid a^{2^{n-1}}=1,\,a^{2^{n-2}}=b^{2},\,a^{b}=a^{-1}\rangle $;
\item $T_n=\langle
a,b \mid a^{2^{n-1}}=1=b^{2},\, a^{b}=a^{1+2^{n-2}}\rangle, \ n\geq 4 $
\item $W_n=\langle
a,b \mid a^{2^{n-1}}=1=b^{2},\,a^{b}=a^{2^{n-2}-1}\rangle$;
\item $R_n = \langle a,b \mid b^3 = 1 = a^{3^{n-1}}, a^b = a^{1+3^{n-2}} \rangle$
\end{itemize}

\begin{lemma} \label{Dedekind}
Let $G$ be a Dedekind group with $\lambda(G)=6$. Then $G$ is abelian. 
\end{lemma}
\begin{proof}
Argue by contradiction and suppose that there exists a non-abelian Dedekind group $G$ in which $\lambda(G)=6$. By \cite[p. 139]{Rob}, the group $G$ is isomorphic to $$G \cong Q_8 \times \underbrace{C_2 \times \ldots \times C_2}_{n \ times} \times A,$$ for some non-negative integer $n$ and $A$ is an abelian group of odd order. Since $\lambda(Q_8)=3$, it follows that either $n \geqslant 1$ or $A$ is a non-trivial group. On the other hand, if $n \geqslant 1$ then $G$ contains a subgroup $H \cong Q_8 \times C_2$ and $6 = \lambda(G) \geqslant \lambda(H) = 8$, a contradiction. Thus $A$ is necessarily non-trivial. By Proposition \ref{prop.basic}, $6 = \lambda(G) = \lambda(Q_8) \times \lambda(A) = 3 \lambda(A)$, which is not our case. Thus, $G$ is abelian. 
\end{proof}

\begin{proposition} (Rog\'erio, \cite[Proposition 7]{JRR}) \label{maximal.cyclic.rog}

\begin{enumerate}
\item[(i)] $\lambda (D_{n})=n+1$;

\item[(ii)] $\lambda (Q_{2^{n}})=2^{n-2}+1$;

\item[(iii)] $\lambda (T_n)=n+1$; 
\item[(iv)] $\lambda (W_n)=2^{n-3}+2^{n-2} + 1$.
\end{enumerate}
\end{proposition}

\begin{lemma} \label{lem.Rn}
Let $n \geqslant 3$. If $R_n = \langle a,b \mid b^3 = 1 = a^{3^{n-1}}, a^b = a^{1+3^{n-2}} \rangle$, then $\lambda(R_n) = 2n$.   
\end{lemma}
\begin{proof}
One observes that the subgroups $\langle a \rangle$, $\langle b\rangle$ and $\langle a^{3^i}b^j\rangle$ are the maximal cyclic subgroup of $R_n$, where $i \in \{0,1,2, \ldots, n-2\}$ and $j \in \{1,2\}$. It follow that $\lambda(R_n) = 2(n-1) + 2 = 2n$.   
\end{proof}

\begin{proposition} \label{prop.p.groups.6}
Let $G$ be a $p$-group with $\lambda(G)=6$. Then $G$ is isomorphic to one the following $C_4 \times C_4$, $C_5 \times C_5$, $R_3$, $C_3 \times C_9$, $C_{16} \times C_2$ or $T_5$. 
\end{proposition}

\begin{proof}
Without loss of generality we can assume that $\overline{G}$ is not cyclic (Proposition \ref{prop.basic} $(ii)$) and we deduce that $|\overline{G}| \in \{16,25,27,32\}$ (Proposition \ref{prop.list} and Lemma \ref{lem.grp.6}). 

Assume that $\vert\overline{G}\vert=16$. By Lemma \ref{lem.grp.6}, $\overline{G} = C_4 \times C_4$. Moreover, we deduce that $|G/N:H_i/N|=4$, $i=1,2,3,4,5,6$ and so, every maximal cyclic is a normal subgroup. Consequently, $G$ is a Dedekind group. By Lemma \ref{Dedekind}, $G$ is abelian. Then $G$ is isomorphic to one of the following $C_{2^{\alpha-2}} \times C_2 \times C_2$, $C_{2^{\alpha-2}} \times C_4$ or  $C_{2^{\alpha-1}} \times C_2$. By Remark \ref{rem.abelian}, 
\begin{itemize}
\item $\lambda(C_{2^{\alpha-2}} \times C_2 \times C_2) = 2^2(\alpha-2) + 2 + (\alpha -2) + 2 = 5\alpha -6 \neq 6$;
\item $\lambda(C_{2^{\alpha-2}} \times C_4) = 4(\alpha-3) - 2(\alpha-5) = 6$ and so, $\alpha=4$ and $G=C_4 \times C_4$;
\item $\lambda(C_{2^{\alpha-1}} \times C_2) = (\alpha-1)+2 = 6$ and so, $\alpha=5$ and $G \cong C_{16} \times C_2$.  
\end{itemize}

Suppose that $|\overline{G}| = 25$. Since $\lambda(G) = \lambda(\overline{G})=6$, it follows that $\overline{G} = C_5 \times C_5$. In particular, $G$ is a $5$-group and has a cyclic subgroup of index $5$. Arguing as in the previous paragraph, we deduce that $G$ is abelian. Hence $G \cong C_{5^{\alpha-1}} \times C_5$. According to Remark \ref{rem.abelian}, $\lambda(G) = 4(\alpha -1)+2 = 6$ and so, $\alpha = 2$. Thus, $G \cong C_5 \times C_5$. 

Assume that $|\overline{G}| = 27$.  By Lemma \ref{lem.grp.6}, $\overline{G}$ is isomorphic to $C_3 \times C_9$ or $\langle a,b \mid a^9=1=b^3, a^b=a^4\rangle$. In both cases, $\overline{G}$ contains a maximal cyclic subgroup of index $3$ and so, $G$ contains a cyclic subgroup of index 3. In particular, we can deduce that the group $G$ is isomorphic to either $C_{3^{n-1}} \times C_3$ or $R_n$, where $n \geqslant 3$ (see \cite[5.3.4]{Rob} for more details). 
\begin{itemize}
\item If $\lambda(C_{3^{n-1}} \times C_3) = (\alpha-1)\cdot 3 - (\alpha -1) + 2 = 6$, then $\alpha=3$ and so, $G \cong C_9 \times C_3$ (Remark \ref{rem.abelian});
\item If $\lambda(R_n) = 2n$, then $n=3$ and so, $G$ is isomorphic to $R_3$ (Lemma \ref{lem.Rn}). 
\end{itemize}

Suppose that $\vert\overline{G}\vert=32$. By Lemma \ref{lem.grp.6}, $\overline{G} \cong C_{16} \times C_2$ or $\overline{G} \cong T_5$. In both cases, $\overline{G}$ contains a cyclic subgroup of index $2$ and so, $G$ contains a maximal cyclic subgroup of index $2$. By \cite[5.3.4]{Rob}, $G$ is isomorphic to one of the following, $C_{2^{\alpha}}$, $C_{2^{\alpha-1}} \times C_2$, $D_{2^{\alpha-1}}$, $Q_{2^{\alpha}}$, $T_{\alpha}$ or $W_{\alpha}$. By Lemma \ref{maximal.cyclic.rog}
\begin{itemize}
\item If $\lambda(C_{2^{\alpha-1}} \times C_2) = 2(\alpha-1)-(\alpha-1)+2=6$, then $\alpha=5$ and so, $G$ is isomorphic to $C_{16} \times C_2$;
\item If $\lambda(T_{\alpha}) = \alpha + 2=6$, then $\alpha=4$ and so, $G$ is isomorphic to $T_5$;
\item Note that $\lambda(D_{2^{\alpha-1}}) = 2^{\alpha-1} + 1=6$ does not happen;
\item Note that $\lambda(Q_{2^{\alpha}}) = 2^{\alpha-2} + 1=6$ does not happen;
\item Note that $\lambda(W_{\alpha}) = 2^{\alpha-3} + 2^{\alpha-2} + 1=6$ does not happen.
\end{itemize}
The result follows. 
\end{proof}

For reader's convenience we restate Theorem B: \\ 

{\noindent}{\bf Theorem B.}{ Let $G$ be a nilpotent group with $\lambda(G) = 6$. Then $G \cong P \times C_n$, where $P$ is a Sylow $p$-subgroup and $p$ does not divides $n$. Moreover, $G$ is isomorphic to one of the following: 
\begin{itemize}
 \item[$(a)$] $C_4 \times C_4 \times C_n$, where $n$ is an odd positive integer;
 \item[$(b)$] $C_5 \times C_5 \times C_n$, where $5$ does not divide $n$;
 \item[$(c)$] $C_3 \times C_9 \times C_n$, where $3$ does not divide $n$;
 \item[$(d)$] $R_3 \times C_n$, where $R_3 \cong \langle a,b \mid a^9=1=b^3, a^b=a^4\rangle$ and $3$ does not divides $n$; 
 \item[$(e)$] $G \cong C_2 \times C_{16} \times C_n$, where n is an odd positive integer;
 \item[$(f)$] $G \cong T_5 \times C_n$, where $T_5 = \langle a,b \mid a^{16}=1=b^2, \ a^b = a^9\rangle$ and $n$ is an odd positive integer. 
\end{itemize}
}
\begin{proof}
Set $\vert G\vert  = p_1^{n_1}p_2^{n_2} \ldots p_r^{n_r}$ and $P_1$,$P_2$, \ldots, $P_r$ the Sylow subgroups of $G$. It follows that $G = P_1 \times \ldots \times P_r$. By Proposition \ref{prop.basic}, $$6=\lambda(G) = \lambda(P_1)\lambda(P_2) \ldots \lambda(P_r).$$ Since there is no groups $G$ with $\lambda(G)=2$, we conclude that there exists an integer $j \in \{1,2,\ldots,r\}$ such that $G = P_j \times C_n$, where $\lambda(P_j)=6$ and $p_j$ does not divide $n$. In particular, every $P_j$ is a $p$-subgroup of a type indicated in Proposition \ref{prop.p.groups.6}. The proof is complete. 
\end{proof}

Now we will deal with Theorem C: {\it Let $G$ be a non-nilpotent group with $\lambda(G) = 6$. Then $G$ is isomorphic to one of the following: 
\begin{itemize}
\item[(a)] $\langle a,b \mid a^5 = 1 = b^n, \ a^b = a^4\rangle$, where $4^n \equiv 1$ $(\text{mod \ } 5)$ and $5 \nmid n$;
\item[(b)] $\langle a,b \mid a^5 = 1 = b^n, \ a^b = a^2\rangle$, where $5 \nmid n$ and $4 \mid n$; 
\item[(c)] $\langle a,b \mid a^5 = 1 = b^n, \ a^b = a^3\rangle$, where $5 \nmid n$ and $4 \mid n$.
\item[(d)] $\langle a,b \mid a^{3n} = 1 = b^{6n}, a^b = a^{n+1}, a^3 = b^6 \rangle$, where $3 \nmid (n+1)$;
\item[(e)] $\langle a,b \mid a^{3n} = 1 = b^{6n}, a^b = a^{2n+1}, a^3 = b^6 \rangle$, where $3 \nmid (n+2)$; 
\item[(f)] $\langle a,b \mid a^{3n} = 1 = b^{6n}, a^3 = b^6, a^b = a^{n+1}\rangle$, where $2 \mid n$ and $3 \nmid (n+1)$;
\item[(g)] $\langle a,b \mid a^{3n} = 1 = b^{6n}, a^3 = b^6, a^b = a^{2n+1}\rangle$, where $2 \mid n$ and $3 \nmid (n+2)$. 
\end{itemize} 
} 

It is a well known result that if $G$ is a group in which all Sylow $p$-subgroups of $G$ are cyclic, then $G$ is metacyclic. Moreover, $$G = \left\langle a,b \mid \ a^m=1=b^n, a^b=a^r \right\rangle,$$ 
where $r^n \equiv 1$ (mod $m$), $m$ is odd, $0 \leqslant r < m$, $m$ and $n(r-1)$ are coprime numbers (Holder, Burnside, Zassenhaus' Theorem \cite[10.1.10]{Rob}). Now, we start describing the structure of groups in which all Sylow $p$-subgroups of $G$ are cyclic and $\lambda(G)=6$. 

\begin{lemma} \label{lem.sylow.cyclic}
Let $G$ be a group with $\lambda(G)=6$. Suppose that all Sylow $p$-subgroup are cyclic. Then $G$ is non-nilpotent. Moreover, the group $G$ is isomorphic to one of the following groups:  
\begin{itemize}
\item[(a)] $\langle a,b \mid \ a^5 = 1 = b^n, \ a^b = a^4\rangle$, where $4^n \equiv 1$ $(\text{mod \ }5)$ and $5 \nmid n$;
\item[(b)] $\langle a,b \mid \ a^5 = 1 = b^n, \ a^b = a^2\rangle$, where $5 \nmid n$ and $4 \mid n$; 
\item[(c)] $\langle a,b \mid \ a^5 = 1 = b^n, \ a^b = a^3\rangle$, where $5 \nmid n$ and $4 \mid n$.
\end{itemize}
\end{lemma}

\begin{proof}
Assume that $G$ is nilpotent. Since every Sylow $p$-subgroup of $G$ is cyclic, it follows that  $G$ is cyclic and consequently, $\lambda(G)=1$. So, we can assume that $G$ is non-nilpotent. We deduce that $|\overline{G}| \in \{10,18,20,36\}$ (Proposition \ref{prop.list}).  

Note that if $|\overline{G}| \in \{18,36\}$. In both cases, by Lemma \ref{lem.grp.6}, a Sylow $3$-subgroup of $G$ is not cyclic. In particular, these cases does not occur. So, we can assume that $|\overline{G}| \in \{10,20\}$.   

\noindent{(a).} Assume that $|\overline{G}|=10$. Therefore $\overline{G} \cong D_{5}$. It is clear that, every Sylow $p$-subgroup of $G$ is cyclic. By Holder, Burnside, Zassenhaus' Theorem \cite[10.1.10]{Rob}, $$G = \left\langle a,b \mid \ a^m=1=b^n, a^b=a^r \right\rangle,$$ 
where $r^n \equiv 1$ (mod $m$), $m$ is odd, $0 \leqslant r < m$, and $m$ and $n(r-1)$ are coprime numbers. Since $\lambda (G)=6$, we conclude that $G$ is not cyclic and so $\left\langle b\right\rangle $ cannot be normal in $G$. In particular, $
Z(G)<\left\langle b\right\rangle $. Indeed if $a^{i}b^{j}\in Z(G)$ then $%
a^{i}b^{j}=$ $\left( a^{i}b^{j}\right) ^{b}=\left( a^{i}\right)
^{b}b^{j}=a^{ir}b^{j}$ and $a^{ir-i}=1$. Therefore $m \mid i(r-1)$. Since $%
\gcd(m,n(r-1))=1$, we have that $m \mid i$ and thus $i=0$. Thus $\left\vert
G:\left\langle b\right\rangle \right\vert =m=5$, because $%
N=Z(G)<\left\langle b\right\rangle <G$ and $m$ is odd. Moreover $r=4$ and $%
\langle b^{2}\rangle =Z(G)$. So $G\cong \left\langle
a,b \mid \ a^{5}=b^{n}=1,a^{b}=a^{4}\right\rangle $ where $4^{n}\equiv 1(\!\!\!\!%
\mod 5)$ and $5 \nmid n$, because $5$ and $n(4-1)$ are coprime numbers. \\  

\noindent{(b) and (c)}. Assume that $|\overline{G}|=20$. Then either $G$ is isomorphic to one of the following $C_{20}, C_2 \times C_{10}, D_{10}, Q_{20},  \langle a,b \mid a^5= 1 = b^4, a^b =a^{-1} \rangle$. Note that $\lambda(C_{20})=1$, $\lambda(C_2 \times C_{10}) = 3$, $\lambda(D_{10})=11$. 

Consider $\overline{G} \cong Q_{20} = \langle a,b \mid a^{10}=1, a^5 = b^2, a^b = a^{-1} \rangle$. In this case, we may assume that  $|G: H_1| = 2$ and $|G:H_i|=5$ for each $i \in \{2,3,4,5,6\}$. In particular, every maximal cyclic is also maximal and so, $G$ is unique covered. By \cite{JR}, $|G/Z(G)|=pq$ for some primes $p,q$ and $\langle Z(G),x\rangle$ is cyclic for every $x \in G$. It follows that, $Z(G) \leqslant H_i$, for every $i \in \{1,2,3,4,5,6\}$. In particular, $Z(G) \leqslant N$. On the other hand, $N \leqslant Z(G)$ and $|\overline{G}|=20$, which is impossible. 

Now, we can assume that $\overline{G} \cong \langle a,b \mid a^5= 1 = b^4, a^b =a^{-1} \rangle$. It is clear that every Sylow $p$-subgroup of $G$ is cyclic. By Holder, Burnside, Zassenhaus' Theorem \cite[10.1.10]{Rob}, 

$$G = \left\langle a,b \mid \ a^m=1=b^n, a^b=a^r \right\rangle,$$ 
where $r^n \equiv 1$ (mod $m$), $m$ is odd, $0 \leqslant r < m$, and $m$ and $n(r-1)$ are coprime numbers. Arguing as in the proof of item (a) we deduce that $m=5$, $4 \mid n$, $5 \nmid n$ and $Z(G) \leqslant \langle b\rangle$. Moreover, $N = Z(G)$, because $N \leqslant Z(G)$ and $Z(\overline{G})= \overline{1}$. Since $G$ is not cyclic and $b^2$ is not central, we deduce that $r \in \{2,3\}$. The proof is complete.   
\end{proof}

We are ready to prove Theorem C. 

\begin{proof}[Proof of Theorem C]
Without loss of generality we can assume $\overline{G}$ is not cyclic (Proposition \ref{prop.basic} $(ii)$). By Proposition \ref{prop.list}, $$\vert\overline{G}\vert \in \{10,16,18,25,27,32,36\}.$$ 

Since $G$ is non-nilpotent group and the cases $\vert \overline{G} \vert = 10$ and $20$ was considered in Lemma \ref{lem.sylow.cyclic}, we can assume that $|\overline{G}| \in \{18,36\}$. \\ 

\noindent{(Case $\vert \overline{G} \vert = 18$).} By Lemma \ref{lem.grp.6}, the group  $\overline{G}$ is isomorphic to $S_3 \times C_3$. We can deduce that all maximal cyclic subgroups of $G$ have index $3$ or $6$ (Proposition \ref{prop.basic}). Let $H_1, \ldots, H_6$ be the maximal cyclic subgroups. More precisely, we can write $|G:H_i|=3$, for $i=1,2,3$ and $|G:H_j|=6$, $j=4,5,6$, where $H_4,H_5,H_6$ are normal subgroups of $G$. Then $G = H_1H_4$ and $H_1 \cap H_4 = N$. Set $H_1 =\langle b\rangle$ and $H_4 = \langle a\rangle$, $n = |N|$. Therefore, we can write $G = \langle a,b \mid a^{3n} = 1 = b^{6n}, a^b = a^r, a^3=b^6\rangle$, for some positive integer $r \in \{2,3,\ldots,3n-1\}$. Now, we describe the values to $r$ in terms of $n$. Since $a^b = a^r$, it follows that $a=a^{b^6}=a^{r^6}$ and so, $3n \mid (r^6-1)$. In particular, $3$ divides $r^3-1$ or $r^3+1$. We can deduce that $3$ divides $r-1$ or $r+1$.   Moreover, since $a^3 \in Z(G)$, it follows that $a = (a^b)^3 = (a^r)^3$ and so, $3n$ divides $3(r-1)$. In particular, $r-1 = nt$ for some positive integer $t$. On the other hand, since the order of $a$ is $3n$ and $r \in \{2,3,\ldots,3r-1\}$, it follows that $t=1$ or $2$. Now, assume that $t=1$. Then $3$ divides $n$ or $3$ divides $n+2$. We deduce that:
$$\langle a,b \mid a^{3n} = 1 = b^{6n}, a^b = a^{n+1}, a^3 = b^6 \rangle,$$ where $3 \nmid (n+1)$. It remains to consider the case $t=2$. It follows that, $3$ divides $n$ or $n+1$ and so, 
$$\langle a,b \mid a^{3n} = 1 = b^{6n}, a^b = a^{2n+1}, a^3 = b^6 \rangle,$$ where $3 \nmid (n+2)$. \\ 

\noindent{(Case $\vert \overline{G} \vert = 36$).} By Lemma \ref{lem.grp.6}, the group  $\overline{G}$ is isomorphic to $$\langle a,b,c \mid a^3=1=b^3=c^4, a^c = a^2, b^c =b, a^b=a\rangle.$$ We can deduce that all maximal cyclic subgroups of $G$ have index $3$ or $6$ (Proposition \ref{prop.basic}). Let $H_1, \ldots, H_6$ be the maximal cyclic subgroups. More precisely, we can write $|G:H_i|=3$, for $i=1,2,3$ and $|G:H_j|=6$, $j=4,5,6$, where $H_4,H_5,H_6$ are normal subgroups of $G$. Set $H_1 =\langle b\rangle$ and $H_4 = \langle a\rangle$, $n = |N|$. Therefore, we can write $$G = \langle a,b \mid a^{3n} = 1 = b^{6n}, a^b = a^r, a^3=b^6\rangle,$$ for some positive integer $r \in \{2,3,\ldots,3n-1\}$ and $2$ divides $n$. Now, we describe the values to $r$ in terms of $n$. Arguing as in the Case $|\overline{G}|=18$ we can deduce that 
\begin{itemize}
\item $3$ divides $r-1$ or $r+1$; 
\item $r-1 = nt$ for some positive integer $t$. 
\end{itemize}
On the other hand, since the order of $a$ is $3n$ and $r \in \{2,3,\ldots,3r-1\}$, it follows that $t=1$ or $2$. Now, assume that $t=1$. Then $3$ divides $n$ or $3$ divides $n+2$. We deduce that:
$$\langle a,b \mid a^{3n} = 1 = b^{6n}, a^b = a^{n+1}, a^3 = b^6 \rangle,$$ where $3 \nmid (n+1)$ and $2$ divides $n$. It remains to consider the case $t=2$. It follows that, $3$ divides $n$ or $n+1$ and so, 
$$\langle a,b \mid a^{3n} = 1 = b^{6n}, a^b = a^{2n+1}, a^3 = b^6 \rangle,$$ where $3 \nmid (n+2)$ and $2$ divides $n$. The result follows. 
\end{proof}

\section{Solvability criteria}

In this section we prove some solvability criteria for group in terms of $\lambda(G)$. We need the following result, due to G. Zhang \cite{Zhang}. 

\begin{theorem} There are no self-normalizing cyclic subgroups in finite simple non-abelian groups.
\label{Zhang}
\end{theorem}

The next well known result is a consequence of Cayley's Theorem.

\begin{lemma} \label{normalizer} Let $G$ be a simple non-abelian group and let $H$ be a proper subgroup of finite index in $G$. Then $G$ is isomorphic to a subgroup of the Alternating group on the left coset space $G/H$.
\end{lemma}

\begin{proof} By Cayley's Theorem  \cite[1.6.8 and 1.6.9]{Rob} the simple group $G \cong G/Core(H)$ is isomorphic to a subgroup $N$ of the symmetric group $S_{|G|/|H|}=Sym(G/H)$. By the simplicity of $N$ and normality of the Alternating group $A_{G/H}=Alt(G/H)$ in $Sym(G/H)$, we have two cases to consider. Note that $Alt(G/H) \cap N$ is normal in $N$ and $N$ is simple, hence $Alt(G/H) \cap N = N$ or $Alt(G/H) \cap N$ is trivial. In the first case, $N \leq Alt(G/H)$ and the result follows. Now if $Alt(G/H) \cap N$ is trivial, then $N \cong NAlt(G/H)/Alt(G/H)$ which is a simple non-abelian subgroup of the order two group $Sym(G/H)/Alt(G/H)$, a contraction and the result follows. 
\end{proof}

Recall that if $G$ is a group with $\lambda(G)=n$, then $H_1, H_2, \ldots, H_n$ are the set of all maximal cyclic subgroups of $G$ (Proposition \ref{prop.basic}). In particular, by Lemma \ref{lem.Cohn}, the index of $H_2$ in $G$ is at most $n-1$.  Here $A_n$ denotes the Alternating group of degree $n$. We will denote by $M_{11}$ and $M_{12}$ the Mathieu groups.

\begin{lemma} Let $G$ be a simple non-abelian group. Suppose that $G$ is isomorphic to some subgroup of $A_n$, where $n \leq 14$. Then $\lambda(G) \geqslant 31$.

\label{geral}
\end{lemma}

\begin{proof} All simple non-abelian group is a perfect group. Therefore using GAP \cite{GAP} we compute the conjugacy classes of perfect subgroups of $A_n$ for $5 \leq n \leq 14$ and we obtain the following representatives of simple non-abelian subgroups

\begin{itemize}
\item[(1)] $n=5: \ A_5$;
\item[(2)] $n=6: \ A_i$, where $5 \leqslant i \leqslant 6$;
\item[(3)] $n=7: \ A_i$ and $PSL(3,2)$, where $5 \leqslant i \leqslant 7$;
\item[(4)] $n=8: \ A_i$ and $PSL(3,2)$, where $5 \leqslant i \leqslant 8$;
\item[(5)] $n=9: \ A_i$ and $PSL(3,2)$, where $5 \leqslant i \leqslant 9$;
\item[(6)] $n=10: \ A_i$, $PSL(3,2)$ and $PSL(2,8)$, where $5 \leqslant i \leqslant 10$;
\item[(7)] $n=11: \ A_i, PSL(3,2), PSL(2,j), M_{11}$, where $5 \leqslant i \leqslant 11$ and $j \in \{8,11\}$;
\item[(8)] $n=12: \ A_i, PSL(3,2), PSL(2,j), M_{11}$ and $M_{12}$, where $5 \leqslant i \leqslant 12$ and $j\in \{8,11\}$;
\item[(9)] $n=13: \ A_i, PSL(3,k), PSL(2,j), M_{11}, M_{12}$, where $5 \leqslant i \leqslant 13$, $j \in \{8,11\}$ and $k\in \{2,3\}$;
\item[(10)] $n=14: \ A_i, PSL(3,k), PSL(2,j), M_{11}, M_{12}$, where $5 \leqslant i \leqslant 14$, $j=8,11,13$ and $k\in \{2,3\}$.
\end{itemize}

Note that $\lambda(H) \leq \lambda(K)$ for $H \leq K$. We have that $A_5$ is a subgroup of $A_n$, for $n \geq 5$, $M_{11}$, $M_{12}$ and $PSL(2,11)$. Moreover $\lambda(A_5)=31$, $\lambda(PSL(3,2))=57$, $\lambda(PSL(3,3))=1275$, $\lambda(PSL(2,8))=127$, and $\lambda(PSL(2,13))=183$. It follows that there is no $G \leq A_n$ satisfying the hypothesis. The result follows. 
\end{proof}

Combining Theorem \ref{Zhang}, Lemma \ref{geral} and \ref{normalizer} with some properties of $\lambda$ we have the following result.

\begin{proof}[Proof of Theorem D]
Suppose that the claim is false and let $G$ be a counterexample of minimal order. Without loss generality we assume that $G$ is a simple non-abelian group. Otherwise, let $N \neq 1$ be a normal subgroup of $G$. By the properties of $\lambda$ in \cite[Propositions $4$ and $5$]{JRR}), the subgroup $N$ and the quotient $G/N$ are solvable because $|N|, |G/N| < |G|$ and $\lambda(N), \lambda(G/N) \leq \lambda(G)$, hence $G$ is a solvable group, a contradiction. Consequently $G$ is a simple non-abelian group.

Let $H_1, \ldots, H_s$ be the set of all maximal cyclic subgroups, where $s=\lambda(G) \leq 30$. By \cite{Cohn}, the index $[G:H_2] \leq s-1 \leq 29$. Since $G$ is a simple group, we have $[G:H_2] \leq 28$, because $H_2$ cannot be maximal cyclic subgroup (otherwise $G$ will be a solvable group, \cite{PK}).

By the Theorem \ref{Zhang}, we have that the index $[N_G(H_2):H_2] \neq 1$ and by the simplicity of $G$, $2 < [G:N_G(H_2)] \leq 14$, hence setting $H=N_G(H_2)$ in the Lemma \ref{normalizer}, we have $G \leq A_n$ with $n \leq 14$ and the result follows of the Lemma \ref{geral}.
\end{proof}

\begin{rem}\label{rem.solvable}
Note that groups with $\lambda(G)=31$ can be non-solvable. For instance, $\lambda(A_5 \times C_n)=31 = \lambda(S_5 \times C_n)=\lambda(SL(2,5) \times C_n)=31$, for every positive integer $n$ such that $n$ and $30$ are coprime numbers.  
\end{rem}

Note that the number of maximal cyclic subgroups of a noncyclic group $G$ is exactly $\lambda(G)$ (Proposition \ref{prop.basic}, above). In particular, it seems natural to ask what information about $G$ can be deduced from the number of cyclic subgroups of $G$ (see also \cite{GaLi,Tarn} for more details). Here $c(G)$  denotes the number of all cyclic subgroups of $G$. We obtain the following related result.  

\begin{corollary} \label{cor.31}
Let $G$ be a group with $c(G) \leqslant 31$. Then $G$ is solvable.
\end{corollary}

\begin{proof}
There is no loss of generality in assuming that $G$ is not cyclic. In particular, $\lambda(G) < c(G)$. Since $c(G)\leqslant 31$, it follows that $\lambda(G)\leqslant 30$ and so, $G$ is solvable (Theorem D). The result follows.   
\end{proof}

Note that the above bound cannot be improved. For instance, $c(A_5)=32$. 

We conclude this paper by including another extremal variant related to the covering problem.

\begin{rem}
In the recent paper \cite{GaLi} M. Garonzi and first author studied the function $\alpha(G)=c(G)/|G|$. They obtained basic properties of $\alpha(G)$ and its asymptotic behavior. In other words, they characterized $\alpha_{\mathcal{F}}= \max\{\alpha(G)| \ G \in \mathcal{F} \}$ and $m_{\mathcal{F}}= \{G \in \mathcal{F}| \ \alpha(G)=\alpha_{\mathcal{F}} \}$, where $\mathcal{F}$ is some family of finite groups (all finite groups, non-abelian, non-nilpotent, non-solvable, non-supersolvable, $\ldots$). Note that $\lambda(G) < c(G)$ and this suggests a motivation to study a new function $\beta(G)=\lambda(G)/|G|$ whose behavior is $0 < \beta(G) < \alpha(G) \leq 1$. 
\end{rem}


\begin{thebibliography}{10}
\bibitem{AAAH} A. Abdollahi, M.\,J. Ataei, S.\,M.\, Amari and A.\,M. Hassanabadi, {\it Groups with a maximal irredundant $6$-cover}, Communications in Algebra,  \textbf{33}, (2005) 3225--3238.

\bibitem{Abdo1} A. Abdollahi and S. M. Jafarian Amiri, {\it On groups
with an irredundant 7-cover}, Journal of Pure and Applied Algebra, {\bf 209}, (2007)  291--300.

\bibitem{Abdo2} A. Abdollahi and S. M. Jafarian Amiri, {\it Minimal
coverings of completely reducible groups}, Publicationes Mathematicae Debrecen, {\bf 72},
(2008) 167--172.

\bibitem{B} M.\,A. Brodie, {\it Uniquely covered groups}, Algebra Colloquium,  \textbf{10}, (2003) 101--108. 

\bibitem{BFS} R. A. Bryce, V. Fedri and L. Serena, {\it Covering groups
with subgroups}, Bulletin Australian Mathematical Society, \textbf{55}, (1997) 469--476.

\bibitem{Bry1} R. A. Bryce and L. Serena, \newblock A note on minimal
coverings of groups by subgroups, Special issue on group theory,
Journal of Australian Mathematical Society, {\bf 71}, (2001) 159--168.

\bibitem{Cohn} J. H. E. Cohn, {\it On n-sum groups}, Mathematica Scandinavica, \textbf{75}, (1994) 44--58.

\bibitem{GAP} The GAP Group, GAP - \textit{Groups, Algorithms, and Programming}, Version 4.8.10; 2018. (https://www.gap-system.org)

\bibitem{GL} M. Garonzi and A. Lucchini, {\it Irredundant and Minimal Covers of Finite Groups}, Communications in Algebra, {\bf 44}, (2016) 1722--1727.  

\bibitem{GaLi} M. Garonzi and I. Lima, {\it On the number of cyclic subgroups of a finite group}, Bulletin of the Brazilian Mathematical Society, New Series. (2018). https://doi.org/10.1007/s00574-018-0068-x (to appear).

\bibitem{JR} S.\,O. Juriaans and J.\,R. Rog\'erio, {\it On groups whose maximal cyclic subgroups are maximal}, Algebra Colloquium \textbf{17}, (2010) 223--227.

\bibitem{Marioti} A. Mar\'{o}ti, {\it  Covering the symmetric groups with proper subgroups}, Journal of Combinatorial Theory, Series A, {\bf 110}, (2005) 97--111.

\bibitem{PK} V.\,V. Pylaev and N.\,F. Kuzennyi, {\it Finite groups with a cyclic maximal subgroup}, Ukrainian Mathematical Journal, \textbf{28}, (1976) 500--506. 

\bibitem{Rob} D.\,J.\,S. Robinson,
\textit{A course in the theory of groups}, 2nd edition, Springer-Verlag, New York, 1996.

\bibitem{JRR} J.\,R. Rog\'erio, {\it A Note on Maximal Coverings of Groups}, Communications in Algebra, \textbf{42}, (2014) 4498--4508.

\bibitem{S} G. Scorza,
{\it Gruppi che possono pensarsi come somma di tre sottogruppi}, Bollettino della Unione Matematica Italiana, \textbf{5}, 
(1926) 216--218.

\bibitem{Tarn} M. T\u{a}rn\u{a}uceanu, Finite groups with a certain number of cyclic subgroups, Amer. Math. Monthly, {\bf 122}, (2015) 275--276.

\bibitem{T} M.\,J. Tomkinson, {\it Groups covered by finitely many coset or subgroups}, Communications in Algebra, \textbf{15}, (1987) 845--859.

\bibitem{Tomk1} M. J. Tomkinson, {\it Groups as the union of proper
subgroups}, Mathematica Scandinavica, {\bf 81}, (1997) 189--198.

\bibitem{JZhang} J. Zhang, {\it Finite groups as the union of proper
subgroups}, Serdica Mathematical Journal,  {\bf 32}, (2006) 259--268.

\bibitem{Zhang} G. Zhang, \textit{On self-normalizing cyclic subgroups}. Journal of Algebra, {\bf 127}, (1989) 255--258.
\end{thebibliography}
\end{document}